\documentclass[11pt]{article}
\usepackage[top=2cm, bottom=2cm, left=2cm, right=2cm]{geometry}
\usepackage{amssymb}
\usepackage{amsmath}
\usepackage{amsthm}
\usepackage{enumerate}
\usepackage{xcolor}
\usepackage{fullpage}
\allowdisplaybreaks

\newtheorem{theorem}{Theorem}[section]

\newtheorem*{claim*}{Claim}

\newtheorem{theorem MTP}{Mass Transference Principle}
\newtheorem*{theorem MTP*}{Mass Transference Principle}
\newtheorem{theorem K}{Khintchine's Theorem}
\newtheorem*{theorem K*}{Khintchine's Theorem}
\newtheorem{theorem J}{Jarn\'{\i}k's Theorem}
\newtheorem*{theorem J*}{Jarn\'{\i}k's Theorem}
\newtheorem{theorem KJ}{Khintchine--Jarn\'{\i}k Theorem}
\newtheorem*{theorem KJ*}{Khintchine--Jarn\'{\i}k Theorem}

\theoremstyle{definition}

\theoremstyle{remark}

\newtheorem*{remark*}{Remark}

\newtheorem*{notation*}{Notation}

\renewcommand{\Bbb}[1]{\mathbb{#1}}
\newcommand{\N}{{\Bbb N}}         % natural numbers
\newcommand{\Q}{{\Bbb Q}}         % rational numbers

\newcommand{\R}{{\Bbb R}}        % real numbers
\newcommand{\Z}{{\Bbb Z}}         % integer numbers
\newcommand{\cA}{{\cal A}}

\newcommand{\cH}{{\cal H}}

\newcommand{\cK}{{\cal K}}

\newcommand{\cM}{{\cal M}}
\newcommand{\cN}{{\cal N}}

\newcommand{\cQ}{{\cal Q}}

\newcommand{\cU}{{\cal U}}

\newcommand{\diam}{r}

\renewcommand{\le}{\leq}
\renewcommand{\ge}{\geq}

\newcommand{\ba}{\mathbf{a}}
\newcommand{\p}{\mathbf{p}}

\newcommand{\bt}{\mathbf{t}}

\newcommand{\x}{\mathbf{x}}

\newcommand{\btau}{\mathbf{\tau}}

\DeclareMathOperator{\dimh}{dim_H}

\newcommand{\Addresses}{{% additional braces for segregating \footnotesize
  \bigskip
  \footnotesize

  D.~Allen,
  \textsc{School of Mathematics, University of Bristol, Fry Building
Woodland Road, Bristol, BS8 1UG, and the Heilbronn Institute for Mathematical Research, Bristol, United Kingdom}\par\nopagebreak
  \textit{E-mail address:} \texttt{demi.allen@bristol.ac.uk}

  \medskip

  B.~Wang,
  \textsc{School of Mathematics and Statistics, Huazhong University of Science and Technology, Wuhan, 430074, P. R. China}\par\nopagebreak
  \textit{E-mail address:} \texttt{bwei\_wang@hust.edu.cn}

}}

\title{{\sc {A note on weighted simultaneous Diophantine approximation on manifolds}}}

\author{Demi Allen\footnote{supported by the Heilbronn Institute for Mathematical Research} \and {Baowei Wang}}

\date{\today}
\begin{document}

\maketitle
\begin{abstract}
In this note, we present an improvement to a recent result due to Beresnevich, Levesley, and Ward (2021) pertaining to weighted simultaneous Diophantine approximation on manifolds.
\end{abstract}


\bigskip


\maketitle

\section{Introduction and statement of the result}

A fundamental problem in metric Diophantine approximation is to understand the ``size'', specifically the Lebesgue measure, Hausdorff measure, and Hausdorff dimension, of the classical set of \emph{simultaneously $\psi$-well approximable points} in $\R^n$. For a given approximating function $\psi: \N \to \R_{\geq 0}$, that is the set
\[W_n(\psi) = \left\{\x=(x_1,\dots,x_n) \in \R^n: \max_{1 \leq i \leq n}{|qx_i-p_i|}<\psi(q) \text{ for infinitely many } (\p,q) \in \Z^n \times \N \right\}.\]

For monotonic approximating functions, the Lebesgue measure and Hausdorff measure of $W_n(\psi)$ are characterised, respectively, by classical results due to Khintchine \cite{Khintchine 1924, Khintchine 1926} and Jarn\'{\i}k \cite{Jarnik31} from the 1920s and 1930s. In the 1940s, Duffin and Schaeffer constructed a counter-example demonstrating that monotonicity of the approximating function is essential in the one-dimensional version of Khintchine's Theorem and proferred a conjecture concerning the expected result for non-monotonic approximating functions \cite{Duffin-Schaeffer}. A higher-dimensional version of the Duffin-Schaeffer Conjecture, which encapsulates the setting of simultaneous Diophantine approximation in $\R^n$, was formulated by Sprind\v{z}uk in \cite[Chapter 1, Section 8]{Sprindzuk}. In 1990, Pollington and Vaughan \cite{Pollington-Vaughan} proved this conjecture of Sprind\v{z}uk for $n \geq 2$ and, in a recent breakthrough, the original conjecture of Duffin and Schaeffer was proved by Koukoulopoulos and Maynard \cite{Koukoulopoulos-Maynard}. Moreover, in \cite{BV MTP} Beresnevich and Velani demonstrated the equivalence of the Duffin-Schaeffer Conjecture and its appropriate Hausdorff measure counterpart. This equivalence is a consequence of the celebrated Mass Transference Principle proved in \cite{BV MTP}. Taking all of the above results together, we have a fairly complete picture regarding the metric theory of the set $W_{n}(\psi)$.

In recent years, it has become increasingly popular to study various generalisations of the set $W_n(\psi)$ and how they interact with other naturally arising sets such as curves and manifolds. The particular generalisation of $W_n(\psi)$ which will concern us here will be the \emph{weighted simultaneously $\Psi$-approximable points} in $\R^n$. Moreover, we will be interested in studying the Hausdorff dimension of the intersection of this set with manifolds for a particular form of the approximating function~$\Psi$. Throughout, given $X \subset \R^n$ and $s \geq 0$, we will write $\cH^s(X)$ to denote the Hausdorff $s$-measure of~$X$. We will denote by $\dimh{X}$ the Hausdorff dimension of $X$ and we will write $\lambda_n(X)$ to denote the $n$-dimensional Lebesgue measure of $X$. For definitions and properties of Hausdorff measures and dimension, we refer the reader to \cite{Falconer}.

Let $\Psi: \N \to \R_{\geq 0}^n$ be an approximating function such that
\[\Psi(q) = (\psi_1(q),\dots,\psi_n(q))\]
where $\psi_i: \N \to \R_{\geq 0}$ for each $1 \leq i \leq n$. We define the \emph{weighted simultaneously $\Psi$-approximable points} in $\R^n$ as
\[W_n(\Psi) = \left\{\x \in \R^n: |qx_i - p_i| < \psi_i(q), \quad 1 \leq i \leq n, \quad \text{for i.m. } (\p,q) \in \Z^n \times \N\right\}.\]
Throughout, we will frequently use the shorthand ``i.m.''~for ``infinitely many''. When $\Psi(q)=(q^{-\tau_1},\dots,q^{-\tau_n})$ for some $\btau = (\tau_1,\dots,\tau_n) \in \R_{>0}^{n}$, we write $W_n(\btau)$ in place of $W_n(\Psi)$.

When considering manifolds, we will look at them locally on some open subset $\cU \subset \R^d$ and will use the following Monge parameterisation without loss of generality:
\begin{align}\label{4}\cM = \{(x,f(x)): x \in \cU\} \subset \R^n,\end{align}
where $d$ is the dimension of the manifold, $m$ is the codimension of the manifold, i.e. $d+m=n$, and $f: \cU \to \R^m$.

Considering the Hausdorff dimension of the intersection of $W_{n}(\btau)$ with a manifold $\cM$, we are able to establish the following result.

\begin{theorem}\label{Conjecture W}
Let $\cM = \{(x,f(x)): x \in \cU \subset \R^d\} \subset \R^n$ where $f: \cU \to \R^m$ is such that $f \in C^{(2)}$. Here, $d$ is the dimension of the manifold $\cM$ and $m$ is its codimension, so that $n=d+m$. Suppose $\btau = (\tau_1,\dots,\tau_n) \in \R_{>0}^n$ is such that
\[\tau_1 \geq \tau_2 \geq \dots \geq \tau_d \geq \max_{d+1 \leq i \leq n}{\tau_i} \quad \text{and} \quad \sum_{j=1}^{m}{\tau_{d+j}}<1.\]
Then,
\[\dimh{(W_n(\btau) \cap \cM)} \geq \min_{1 \leq i \leq d}\left\{\frac{n+1+\sum_{k=i}^{n}{(\tau_i-\tau_k)}}{\tau_i + 1}-m\right\}.\]
\end{theorem}

Theorem \ref{Conjecture W} improves upon a recent result due to Beresnevich, Levesley, and Ward \cite[Theorem~1.9]{BLW}, who established the same lower bound for the Hausdorff dimension of $W_n(\btau) \cap \cM$ as in Theorem \ref{Conjecture W} but subject to more stringent conditions on the \emph{weight vector} $\btau$. More precisely, they required that $\btau = (\tau_1,\dots,\tau_n) \in \R_{>0}^n$ satisfies
\[\tau_1 \geq \tau_2 \geq \dots \geq \tau_d \geq \max_{d+1 \leq i \leq n}\left\{\tau_i, \frac{1-\sum_{j=1}^{m}{\tau_{j+d}}}{d}\right\} \quad \text{and} \quad \sum_{j=1}^{m}{\tau_{d+j}}<1.\]
In the interest of brevity, we refer the reader to \cite{BLW} and references therein for a detailed overview regarding the state of the art of this problem. Nevertheless, it is still worth remarking here that in the classical case of unweighted simultaneous Diophantine approximation, where the approximating functions in each co-ordinate direction are equal, establishing results about $W_n(\psi) \cap \cM$ is a well-studied problem. In particular, in this case, we note that the problem of establishing upper and lower bounds for the Hausdorff dimension of $W_n(\psi) \cap \cM$ for various manifolds and various forms of the approximating function $\psi$ has been studied in, for example, \cite{B, BDV,  BLVV,  BVVZ, BV2007, BZ2010, Simmons, VV}. Indeed, one may think of Theorem \ref{Conjecture W} and \cite[Theorem 1.9]{BLW} as a generalisation of \cite[Theorem 1]{BLVV} to the setting of weighted simultaneous approximation on manifolds. Results relating to the Lebesgue measure of the set of weighted simultaneously $\Psi$-approximable points, $W_n(\Psi)$, not intersected with manifolds, can be traced back to work of Khintchine \cite{Khintchine weighted}. The Hausdorff dimension of sets of weighted simultaneously approximable points was studied by Rynne in \cite{Rynne}.

Our proof strategy for establishing Theorem \ref{Conjecture W} mirrors that employed by Beresnevich, Levesley, and Ward in proving \cite[Theorem 1.9]{BLW}. In particular, we begin by establishing a ``Dirichlet-type'' theorem for weighted approximation on manifolds (Theorem \ref{Dirichlet Theorem}) and we use this result to construct an appropriate full measure set contained in our manifold $\cM$. Our proof is completed via an application of a mass transference principle from ``rectangles to rectangles'' proved recently by Wang and Wu in \cite{WW2019} which allows us to deduce the desired lower bound on the Hausdorff dimension of $W_n(\tau) \cap \cM$. The key differences allowing for the improvement in Theorem \ref{Conjecture W} compared with \mbox{\cite[Theorem 1.9]{BLW}} are the more general ``Dirichlet-type'' theorem and the use of this more recent mass transference principle from \cite{WW2019} in the final step of our proof. At the analogous point in the proof of \cite[Theorem 1.9]{BLW}, Beresnevich, Levesley and Ward use an earlier mass tranference principle from ``balls to rectangles'' established by Wang, Wu, and Xu in \cite{WWX2015}.

\section{Mass transference principle from rectangles to rectangles}

The original Mass Transference Principle due to Beresnevich and Velani \cite{BV MTP} allows us to infer Hausdorff measure statements for $\limsup$ sets of balls in $\R^n$ from appropriate Lebesgue measure statements. This result is somewhat surprising since Hausdorff measure is in some sense a refinement of Lebesgue measure. As a result, the Mass Transference Principle has a number of profound consequences and has come to be an important tool in the study of metric Diophantine approximation. For example, one such consequence is the aforementioned example that the Mass Transference Principle implies the equivalence of the Duffin-Schaeffer Conjecture and its Hausdorff measure analogue.

Given a real number $s > 0$ and a ball $B := B(x,r)$ in $\R^n$ of radius $r$ centred at $x$, let \mbox{$B^s := B(x,r^{\frac{s}{n}})$}. In particular, note that $B^n = B$. In its most simple form, the Mass Transference Principle reads as follows.

\begin{theorem MTP*}[Beresnevich -- Velani, \cite{BV MTP}]\label{mtp theorem}
Let $\{B_j\}_{j\in\N}$ be a sequence of balls in $\R^n$ with
radii $\diam(B_j)\to 0$ as $j\to\infty$. Let $s>0$ and let $\Omega$ be a ball in $\R^n$. Suppose that, for any ball $B$ in~$\Omega$,
\[ \cH^n\big(\/B\cap\limsup_{j\to\infty}B^s_j{}\,\big)=\cH^n(B) \ .\]
Then, for any ball $B$ in $\Omega$,
\[ \cH^s\big(\/B\cap\limsup_{j\to\infty}B^n_j\,\big)=\cH^s(B) \ .\]
\end{theorem MTP*}

\begin{remark*}
The statement above is a simplified form of \cite[Theorem 2]{BV MTP}. Also recall that for Borel sets in $\R^n$, the Hausdorff $n$-measure, $\cH^n$, is a constant multiple times $n$-dimensional Lebesgue measure,~$\lambda_n$ (see, for example, \cite{Falconer}). So the above statement really does allow us to pass between Lebesgue measure and Hausdorff measure statements.
\end{remark*}

Since the publication of \cite{BV MTP}, the original Mass Transference Principle from ``balls to balls'' stated above has been extended in a multitude of directions. For example, in \cite{BV slicing}, Beresnevich and Velani established a mass transference principle for systems of linear forms which was later improved in \cite{Allen-Beresnevich}. This result was further extended to a mass transference principle for sets satisfying a ``local scaling property'' in locally compact metric spaces in \cite{Allen-Baker}. A mass transference principle from ``balls to rectangles'' was established in \cite{WWX2015} and has subsequently been upgraded to a mass transference principle from ``rectangles to rectangles'' in \cite{WW2019}. Results relating to a mass transference principle from ``balls to arbitrary open sets'' have been established by Koivusalo and Rams \cite{Koivusalo-Rams} and Zhong~\cite{Zhong}.

As hinted at earlier, for our current purposes, we are particularly interested in the mass transference principle from ``rectangles to rectangles'' established in \cite{WW2019}. To aid readability we will not state the result of \cite{WW2019} in full generality but rather we will present here a simplified statement which follows as a corollary of the mass transference principle given by \cite[Theorem 3.4]{WW2019} and which is more directly applicable to the problem at hand.

Let $\cU=\prod_{i=1}^n\cU_i \subset \R^n$ with $\cU_i$ being locally compact in $\mathbb{R}$ for each $1\le i\le n$. For $\ba = (a_1,\dots,a_n) \in \R_{>0}^n$ and a constant $c > 0$, let
\[W^{c}_{\ba} = \left\{x \in \cU: \left|x_i - \frac{p_i}{q}\right|<\frac{c}{q^{a_i}} \quad \text{for } 1 \leq i \leq n, \quad \text{for i.m. } (\p,q) \in \Z^n \times \N \right\}.\]
For $\bt = (t_1,\dots,t_n) \in \R_{\geq0}^n$ and a constant $c' >0$, define
\[W^{c'}_{\ba}(\bt) = \left\{x \in \cU: \left|x_i-\frac{p_i}{q}\right| < \frac{c'}{q^{a_i + t_i}} \quad \text{for } 1 \leq i \leq n, \quad \text{for i.m. } (\p,q) \in \Z^n \times \N\right\}.\]
The following statement is a consequence of \cite[Theorem 3.4]{WW2019}.

\begin{theorem}[Wang -- Wu, \cite{WW2019}] \label{rectangles mtp}
Let $W^{c}_{\ba}$ and $W^{c'}_{\ba}(\bt)$ be as defined above and suppose that\
\[\lambda_n(W^{c}_{\ba}) = \lambda_n(\cU).\]
Then,
\begin{align*}
\dimh{(W^{c'}_{\ba}(\bt))} \geq \min_{A \in \cA}\left\{\sum_{k \in \cK_1 \cup \cK_2}{1}+\frac{\sum_{k \in \cK_3}{a_k}-\sum_{k \in \cK_2}{t_k}}{A}\right\},
\end{align*}
where
\[\cA = \{a_i, a_i+t_i: 1 \leq i \leq n\}\]
and for each $A \in \cA$, the sets $\cK_1, \cK_2, \cK_3$ are defined as
\[\cK_1 = \{k: a_k \geq A\}, \quad \cK_2=\{k: a_k+t_k \leq A\} \setminus \cK_1, \quad \cK_3=\{1,\dots,n\}\setminus (\cK_1 \cup \cK_2)\]
thus giving a partition of $\{1,\dots,n\}$.
\end{theorem}

\section{Proof of Theorem \ref{Conjecture W}}

In this section, we will present our proof of Theorem \ref{Conjecture W}. We follow essentially the same strategy as laid out by Beresnevich, Levesley, and Ward in their proof of \cite[Theorem 1.9]{BLW}. As already alluded to, the main difference is that we use the mass transference principle from ``rectangles to rectangles'' proved in \cite{WW2019}, whereas Beresnevich, Levesley and Ward used the earlier mass transference principle from ``balls to rectangles'' established in \cite{WWX2015}. Our proof is split into three main parts. In Section~\ref{Dirichlet section} we prove a Dirichlet-type theorem (Theorem \ref{Dirichlet Theorem}) for weighted approximation on manifolds. In Section \ref{preparations section} we use Theorem \ref{Dirichlet Theorem} to construct an appropriate full measure set and make some other preliminary preparations which eventually enable us to apply the mass tranference principle for rectangles (Theorem \ref{rectangles mtp}) to complete the proof in Section \ref{proof completion section}.

\subsection{A Dirichlet-type theorem for weighted simultaneous approximation on manifolds} \label{Dirichlet section}

The first step in establishing Theorem \ref{Conjecture W} is to prove a Dirichlet-type theorem for weighted simultaneous approximation on manifolds which will eventually help us to ``construct'' a suitable full measure set to which we can apply Theorem \ref{rectangles mtp}. The following statement provides a suitable Dirichlet-type theorem and is a modification of \mbox{\cite[Theorem 3.1]{BLW}}. Indeed, we follow the same line of proof as the proof of \cite[Theorem 3.1]{BLW} with some parts following (almost) verbatim.

\begin{theorem} \label{Dirichlet Theorem}
Let $\cM = \{(x, f(x)): x \in \cU \subset \R^d\} \subset \R^n$ where $f: \cU \to \R^m$ is such that $f \in C^{(2)}$. Let $\btau = (\tau_1, \dots, \tau_n) \in \R_{>0}^n$ and $\ba = (a_1,\dots,a_d) \in \R_{>0}^d$ be such that
\[\sum_{j=1}^{m}{\tau_{d+j}}<1,\quad {\text{and}}\quad \min_{1 \leq i \leq d}{a_i} > 1.\]
Suppose also that
\begin{align} \label{a condition}
(a_1 - 1) + \dots + (a_d - 1) + \sum_{j=1}^{m}{\tau_{d+j}}=1.
\end{align}
For any $x \in \cU$, there exists an integer $Q_0$ such that for any $Q > Q_0$ there exists \mbox{$(p_1,\dots,p_n,q) \in \Z^n \times \N$} with $1 \leq q \leq Q$ and $\left(\frac{p_1}{q}, \dots, \frac{p_d}{q}\right) \in \cU$ such that
\begin{align} \label{Dirichlet1}
\left|x_i - \frac{p_i}{q}\right| < \frac{4^{m/d}}{qQ^{a_i-1}} \quad \text{for } 1 \leq i \leq d,
\end{align}
and
\begin{align} \label{Dirichlet2}
\left|f_j\left(\frac{p_1}{q},\dots,\frac{p_d}{q}\right)-\frac{p_{d+j}}{q}\right| < \frac{1}{2q^{\tau_{j+d}+1}} \quad \text{for } 1 \leq j \leq m.
\end{align}
Furthermore, for any $x \in \cU \setminus \Q^d$, there exist infinitely many tuples $(p_1,\dots,p_n,q) \in \Z^n \times \N$ with $\left(\frac{p_1}{q},\dots,\frac{p_d}{q}\right) \in \cU$ satisfying \eqref{Dirichlet2} and
\begin{align} \label{Dirichlet3}
\left|x_i - \frac{p_i}{q}\right| < \frac{4^{m/d}}{q^{a_i}} \quad \text{for } 1 \leq i \leq d.
\end{align}
\end{theorem}

\begin{proof}
Since $\cM$ is constructed via a twice continuously differentiable function $f: \cU \to \R^m$, we can choose a suitable $\cU$ such that, without loss of generality, the following two constants exist: \label{C definition}
\[C = \max_{\substack{1 \leq i,k \leq d \\ 1 \leq j \leq m}}{\sup_{x \in \cU}{\left|\frac{\partial^2 f_j}{\partial x_i \partial x_k}(x)\right|}} < \infty,\]
and
\[D = \max_{\substack{1 \leq i \leq d \\ 1 \leq j \leq m}}{\sup_{x \in \cU}{\left|\frac{\partial f_j}{\partial x_i}(x)\right|}} < \infty.\]

For $1 \leq j \leq m$, define
\[g_j = f_j - \sum_{i=1}^{d}{x_i \frac{\partial f_j}{\partial x_i}}.\]
Consider the system of inequalities,
\begin{align}
\left|qg_j(x)+\sum_{i=1}^{d}{p_i\frac{\partial f_j}{\partial x_i}(x)} - p_{d+j}\right| &< \frac{Q^{-\tau_{j+d}}}{4} \quad \text{for } 1 \leq j \leq m, \label{Dirichlet4} \\
|qx_i - p_i| &< \frac{4^{m/d}}{Q^{a_i-1}}, \quad \text{for } 1 \leq i \leq d, \label{Dirichlet5} \\
|q| \leq Q. \label{Dirichlet6}
\end{align}

It is a consequence of Minkowski's Theorem for Systems of Linear Forms (see, for example, \mbox{\cite[Theorem~2C]{Schmidt}}) that there exists a non-zero integer solution $(p_1,\dots,p_n,q) \in \Z^{n+1}$ satisfying the inequalities \eqref{Dirichlet4}--\eqref{Dirichlet6}. More precisely, consider the matrix
\[A=
\begin{pmatrix}
g_1 & \frac{\partial f_1}{\partial x_1} & \dots & \frac{\partial f_1}{{\partial}x_d} & -1 & \dots & 0 \\
\vdots & \vdots & & \vdots & \vdots  & \ddots & \vdots \\
g_m & \frac{\partial f_m}{\partial x_1} & \dots & \frac{\partial f_m}{\partial x_d} & 0 & \dots & -1 \\
x_1 & -1 & \dots & 0 & 0 & \dots & 0 \\
\vdots & \vdots & \ddots & \vdots & \vdots &\ddots &\vdots \\
x_d & 0 & \dots & -1 & 0 & \dots & 0 \\
1 & 0 & \dots & \dots\ & \dots & \dots & 0
\end{pmatrix}.\]
It is straightforward to verify that $|\det(A)|=1$. Next we consider the product of the terms on the right-hand sides of \eqref{Dirichlet4}, \eqref{Dirichlet5}, and \eqref{Dirichlet6}. Recalling condition \eqref{a condition}, we have
\begin{align*}
\left(\prod_{j=1}^{m}{\frac{Q^{-\tau_{j+d}}}{4}}\right)\left(\prod_{i=1}^{d}{\frac{4^{m/d}}{Q^{a_i-1}}}\right) \times Q &= 4^{-m} \cdot \frac{1}{Q^{\tau_{d+1}}} \cdots \frac{1}{Q^{\tau_{d+m}}} \cdot (4^{m/d})^d \cdot \frac{1}{Q^{a_1-1}} \cdots \frac{1}{Q^{a_d-1}} \cdot Q \\
                                               &= \frac{Q}{Q^{(a_1-1)+\dots+(a_d-1)+\sum_{j=1}^{m}{\tau_{d+j}}}} \\
                                               &=1.
\end{align*}
Thus, it follows from Minkowski's Theorem for Systems of Linear Forms that there exists a non-zero integer solution $(p_1,\dots,p_n,q) \in \Z^{n+1}$ satisfying the inequalities \eqref{Dirichlet4}--\eqref{Dirichlet6}, as required.

Now, fix some $x \in \cU$. Then, since $\cU$ is open, there exists a ball $B(x,r)$ for some $r > 0$ which is contained in $\cU$. Define
\[\cQ = \left\{Q \in \N: \max_{1 \leq i \leq d}{\left(\frac{4^{m/d}}{Q^{a_i-1}}\right)} < \min\left\{1,r,\left(\frac{1}{2Cd^2}\right)^{\frac{1}{2}}\right\}\right\}.\]
Since $a_i > 1$ for all $1 \leq i \leq d$, we must have
\[\max_{1 \leq i \leq d}{\left(\frac{4^{m/d}}{Q^{a_i-1}}\right)} \to 0 \quad \text{as} \quad Q \to \infty.\]
Thus, there exists some integer $Q_0$ such that $Q \in \cQ$ for any integer $Q \geq Q_0$. We will show that for any $Q \in \cQ$, a solution $(p_1,\dots,p_n,q) \in \Z^{n+1}$ satisfying inequalities \eqref{Dirichlet4}--\eqref{Dirichlet6} also satisfies inequalities \eqref{Dirichlet1} and~\eqref{Dirichlet2}.

Let us first deal with the case where $q = 0$. By the definition of the set $\cQ$, we have
\[\max_{1 \leq i \leq d}{\left(\frac{4^{m/d}}{Q^{a_i-1}}\right)} < 1.\]
Thus, if $q=0$, it follows from \eqref{Dirichlet5} that $|p_i| < 1$, and hence we must have $p_i=0$, for all $1 \leq i \leq d$. Subsequently, in this case, \eqref{Dirichlet4} yields
\[|p_{d+j}| < \frac{Q^{-\tau_{j+d}}}{4} < 1 \quad \text{for } 1 \leq j \leq m.\]
Hence, in the case when $q=0$, our solution must be the zero vector, thus contradicting Minkowski's Theorem for Systems of Linear Forms. Hence we must have $|q| \geq 1$ and, without loss of generality, we will suppose from now on that $q \geq 1$.

Dividing \eqref{Dirichlet5} by $q$, we obtain
\[\left|x_i - \frac{p_i}{q}\right| < \frac{4^{m/d}}{qQ^{a_i-1}} \quad \text{for } 1 \leq i \leq d,\]
which coincides precisely with \eqref{Dirichlet1}. Moreover, note that it follows from the definition of $\cQ$ that $\left(\frac{p_1}{q}, \dots, \frac{p_d}{q}\right) \in B(x,r) \subset \cU$.

Now we need to prove that the solution $(p_1,\dots,p_n,q) \in \Z^{n+1}$ satisfying inequalities \eqref{Dirichlet4}--\eqref{Dirichlet6} also satisfies~\eqref{Dirichlet2}. This conclusion follows via Taylor's Approximation Theorem (see, for example, \mbox{\cite[Chapter 9]{Rudin}}). In detail, since $\left(\frac{p_1}{q}, \dots, \frac{p_d}{q}\right) \in B(x,r) \subset \cU$ and $f: \cU \to \R^m$ is twice differentiable, we can use the second order case of Taylor's Theorem in $d$ dimensions to conclude that
\begin{align} \label{Taylor}
f_j\left(\frac{p_1}{q},\dots,\frac{p_d}{q}\right) = f_j(x) + \sum_{i=1}^{d}{\frac{\partial f_j}{\partial x_i}(x) \left(\frac{p_i}{q}-x_i\right)} + R_j(x,\hat{x})
\end{align}
for some $\hat{x}$ on the line connecting $x$ and $\left(\frac{p_1}{q},\dots,\frac{p_d}{q}\right)$ where
\[R_j(x,\hat{x}) = \frac{1}{2}\sum_{i=1}^{d}{\sum_{k=1}^{d}{\frac{\partial^2 f_j}{\partial x_i \partial x_k}(\hat{x})\left(\frac{p_i}{q}-x_i\right)\left(\frac{p_k}{q}-x_k\right)}}.\]

Combining the above with our definition of $g_j$, we can rewrite the left-hand side of \eqref{Dirichlet4} as
\begin{align} \label{rewriting}
	&\phantom{=}\left|qg_j(x)+\sum_{i=1}^{d}{p_i\frac{\partial f_j}{\partial x_i}(x)} - p_{d+j}\right| \nonumber \\
	&= \left|q\left(f_j\left(\frac{p_1}{q},\dots,\frac{p_d}{q}\right) + \sum_{i=1}^{d}{\frac{\partial f_j}{\partial x_i}(x)\left(x_i-\frac{p_i}{q}\right) - R_j(x,\hat{x}) - \sum_{i=1}^{d}{x_i \frac{\partial f_j}{\partial x_i}(x)}}\right)+\sum_{i=1}^{d}{p_i \frac{\partial f_j}{\partial x_i}(x)} - p_{d+j} \right| \nonumber \\
	&= \left|qf_j\left(\frac{p_1}{q},\dots,\frac{p_d}{q}\right) - p_{d+j} - qR_j(x,\hat{x})\right|.
\end{align}

Suppose for a moment that
\begin{align} \label{R assumed bound}
|qR_j(x,\hat{x})| < \frac{q^{-\tau_{j+d}}}{4} \quad \text{for } 1 \leq j \leq m.
\end{align}

With this assumption, it follows from the reverse triangle inequality combined with \eqref{rewriting}, \eqref{Dirichlet4}, and~\eqref{Dirichlet6}, that
\begin{align*}
\left|qf_j\left(\frac{p_1}{q},\dots,\frac{p_d}{q}\right) - p_{d+j}\right| &\leq \left|qg_j(x) + \sum_{i=1}^{d}{p_i \frac{\partial f_j}{\partial x_i}(x)} - p_{d+j}\right| + |qR_j(x,\hat{x})| \\
                                                                         &\leq \frac{Q^{-\tau_{j+d}}}{4}+\frac{q^{-\tau_{j+d}}}{4} \\
                                                                         &\leq \frac{q^{-\tau_{j+d}}}{2}.
\end{align*}
Dividing through by $q$, we obtain,
\[\left|f_j\left(\frac{p_1}{q},\dots,\frac{p_d}{q}\right)-\frac{p_{d+j}}{q}\right| \leq \frac{1}{2q^{\tau_{j+d}+1}},\]
which is precisely the inequality given in \eqref{Dirichlet2} which we were trying to establish. Thus, to complete this part of the proof, it remains to verify that we can indeed assume \eqref{R assumed bound}. To this end, write $a:=\min_{1 \leq i \leq d}{a_i}$. By hypothesis, we have $a>1$. Upon recalling the definition of our constant $C$ defined on page \pageref{C definition}, it follows from the definition of $R_j(x, \hat{x})$  and \eqref{Dirichlet5} that, for each $1 \leq j \leq m$, we have
\begin{align*}
|q R_j(x,\hat{x})| &= \left|q \cdot \frac{1}{2} \sum_{i=1}^{d}{\sum_{k=1}^{d}{\left(\frac{\partial^2 f_j}{\partial x_i \partial x_k}(\hat{x})\left(\frac{p_i}{q}-x_i\right)\left(\frac{p_k}{q}-x_k\right)\right)}}\right| \\
                   &\leq \left|\frac{q}{2}\right|\sum_{i=1}^{d}{\sum_{k=1}^{d}{\left(C\left|\frac{p_i}{q}-x_i\right|\left|\frac{p_k}{q}-x_k\right|\right)}} \\
                   &\le \frac{qCd^2}{2}\left(\frac{4^{m/d}}{qQ^{a-1}}\right)^2 \\
                   &=\frac{Cd^24^{2m/d}}{2}\cdot \frac{1}{q}\cdot \frac{1}{Q^{2(a-1)}}.
\end{align*}
Since $\tau_{j+d}<1$, the inequality (\ref{R assumed bound}) would follow if we could show
\begin{align}\label{1}
\frac{Cd^24^{2m/d}}{2}\cdot \frac{1}{Q^{2(a-1)}}<\frac{1}{4}.
\end{align}
Now, since $Q\in \cQ$, the inequality (\ref{1}) follows immediately from the definition of $\cQ$. Hence this part of the proof is complete.

Finally, it remains to prove the second part of the theorem, that for any $x \in \cU \setminus \Q^d$, there exist infinitely many tuples $(p_1,\dots,p_n,q) \in \Z^n \times \N$ with $\left(\frac{p_1}{q},\dots,\frac{p_d}{q}\right) \in \cU$ which satisfy \eqref{Dirichlet2} and \eqref{Dirichlet3}. We will argue by contradiction so let us suppose for a moment that there are only finitely many solutions and denote by $\cA$ the set of all solutions $(p_1,\dots,p_n,q) \in \Z^{n+1}$. Since $x \in \cU \setminus \Q^d$, there must be some $1 \leq k \leq d$ such that $x_{k} \notin \Q$. For such~$k$, there exists some $\delta>0$ such that
\begin{align}\label{3}\delta \leq \min_{(p_1,\dots,p_n,q) \in {\cA}}{|qx_{k}-p_{k}|}.\end{align}
Since each $Q \in \cQ$ corresponds to a solution in $\cA$, and $\cQ$ is infinite while $\cA$ is finite, there must exist one solution $(p_1,\dots,p_n,q) \in \cA$ which corresponds to infinitely many $Q$, say $\{Q_{\ell}\}_{{\ell}\ge 1}$. Thus, for all $\ell \geq 1$, there exists $1 \leq q \leq Q_{\ell}$ such that $$
|qx_k-p_k|<\frac{4^{m/d}}{Q_{\ell}^{a_k-1}} \leq \frac{4^{m/d}}{q^{a_k-1}}.
$$ Since $a_k>1$, the middle term above tends to $0$ which yields a contradiction to (\ref{3}). Thus, we conclude that for any $x \in \cU \setminus \Q^d$ there are infinitely many solutions to \eqref{Dirichlet3} also satisfying \eqref{Dirichlet2}.
\end{proof}

\subsection{Preparing to use the mass transference principle for rectangles} \label{preparations section}

We now turn our attention to constructing a suitable full measure $\limsup$ set to which we will later apply the mass transference principle for rectangles (Theorem \ref{rectangles mtp}). Following again the same line of reasoning as in \cite{BLW}, we begin by defining
\begin{align*}
\cN(f,\btau) = &\left\{(p_1,\dots,p_d,q) \in \Z^{d+1}: \left(\frac{p_1}{q},\dots,\frac{p_d}{q}\right) \in \cU \text{ and} \ \left\|qf\left(\frac{p_1}{q},\dots,\frac{p_d}{q}\right)\right\|<\frac{1}{2q^{\tau_{j+d}}}, \ \ 1 \leq j \leq m\right\}.
\end{align*}

Note that from Theorem \ref{Dirichlet Theorem} we have that for all $x \in \cU \setminus \Q^d$, that is for almost all $x \in \cU$, there are infinitely many different vectors $(p_1,\dots,p_d,q) \in \cN(f,\btau)$ for which
\[\left|x_i-\frac{p_i}{q}\right| < \frac{4^{m/d}}{q^{a_i}}, \quad \text{for } 1 \leq  i \leq d,\]
where
\[(a_1-1)+\dots+(a_d-1)+\sum_{j=1}^{m}{\tau_{d+j}}=1, \quad \text{and} \quad \min_{1 \leq i \leq d}{a_i} > 1.\]

Writing
\[B_{(\p,q)}(c) = \left\{x \in \cU: \left|x_i - \frac{p_i}{q}\right|<\frac{c}{q^{a_i}} \quad \text{for } 1 \leq i \leq d \right\}\]
with $c=4^{m/d}$, by Theorem \ref{Dirichlet Theorem} we have
\[\lambda_d\left(\limsup_{(\p,q) \in \cN(f,\btau)}{B_{(\p,q)}(c)}\right) = \lambda_d(\cU).\]

Eventually we will apply the mass transference principle for rectangles (Theorem \ref{rectangles mtp}) to the full measure $\limsup$ set
\[\limsup_{(\p,q) \in \cN(f,\btau)}{B_{(\p,q)}(c)}.\]
First, however, we will establish some other necessary preliminaries. For some constant $c'>0$, let us write
\[W':=\left\{(x,f(x)) \in \cU \times \R^m: x \in \prod_{i=1}^{d}{B\left(\frac{p_i}{q},\frac{c'}{q^{\tau_i+1}}\right)} \quad \text{for i.m. } (p_1,\dots,p_d,q) \in \cN(f,\btau)\right\}.\]
For $c'$ sufficiently small such that $$
Dc'd<\frac{1}{4}
$$ we will show that
\begin{align} \label{manifold containment}
W' &\subset W_n(\btau)\cap\cM.
\end{align}
The consequence of this is that it will be sufficient for us to consider the Hausdorff dimension of the set
\begin{align}\label{2}W:=\left\{x \in \cU: x \in \prod_{i=1}^{d}{B\left(\frac{p_i}{q},\frac{c'}{q^{\tau_i+1}}\right)} \quad \text{for i.m. } (p_1,\dots,p_d,q) \in \cN(f,\btau) \right\}.\end{align}
To see this, suppose for a moment that \eqref{manifold containment} is true. Then, a lower bound on the Hausdorff dimension of the set $W'$ will automatically yield a lower bound for $\dimh{(W_n(\btau)\cap\cM)}$, which is what we are actually interested in. Next, let
\begin{align*}
\pi_d(W') &= \{x\in \cU:(x,f(x)) \in W'\} \\
          &= \left\{x\in \cU: \left|x_i-\frac{p_i}{q}\right| < \frac{c'}{q^{\tau_i+1}}, \quad 1 \leq i \leq d, \quad \text{for i.m. } (p_1,\dots,p_d,q) \in \cN(f,\tau)\right\}
\end{align*}
be the orthogonal projection of $W'$ onto $\R^d$ and note that $\pi_d(W') = W$. We recall that Hausdorff dimension is preserved under bi-Lipschitz mappings (see, for example, \cite[Corollary 2.4]{Falconer}). Since $\pi_d$, being an orthogonal projection, is such a mapping, we have
\[\dimh{W} = \dimh{\pi_d(W')} = \dimh{W'}\]
and so we conclude that it is sufficient for us to find a lower bound for the Hausdorff dimension of the set~$W$.

In order to show the inclusion in \eqref{manifold containment}, let us suppose that $(x,f(x)) \in W'$ and let \mbox{$(p_1,\dots,p_d,q) \in \cN(f,\btau)$} be such that
\[x \in \prod_{i=1}^{d}{B\left(\frac{p_i}{q},\frac{c'}{q^{\tau_i+1}}\right)}.\]
To prove the containment, we need to show that for such $x$, we also have
\[\|qf_j(x)\| < q^{-\tau_{j+d}} \quad \text{for } 1 \leq j \leq m.\]
To do this, we will again use Taylor's Approximation Theorem. Recalling our earlier definitions of the constants $C$ and $D$ (see page \pageref{C definition}),
together with the definition of the set $\cN(f,\btau)$, and the assumption that $\tau_1 \geq \dots \geq \tau_d$, it follows from \eqref{Taylor} that
\begin{align*}
\|qf_j(x)\| &\leq \left\|qf_j\left(\frac{p_1}{q},\dots,\frac{p_d}{q}\right)\right\| + \left|\sum_{i=1}^{d}{\frac{\partial f_j}{\partial x_i}(x)(qx_i - p_i)}\right| + \left|\frac{1}{q}\frac{1}{2}\sum_{i=1}^{d}{\sum_{k=1}^{d}{\frac{\partial^2 f_j}{\partial x_i \partial x_k}(\hat{x})(p_i-qx_i)(p_k-qx_k)}}\right| \\
            &\leq \frac{1}{2q^{\tau_{j+d}}} + \sum_{i=1}^{d}{\left|\frac{\partial f_j}{\partial x_i}(x)\right||p_i-qx_i|} + \frac{1}{2q}\sum_{i=1}^{d}{\sum_{k=1}^{d}{\left|\frac{\partial^2 f_j}{\partial x_i \partial x_k}(\hat{x})\right||p_i-qx_i||p_k-qx_k|}} \\
            &\leq \frac{1}{2q^{\tau_{j+d}}} + Dc'\sum_{i=1}^{d}{q^{-\tau_i}} + \frac{C(c')^2}{2q}\sum_{i=1}^{d}{\sum_{k=1}^{d}{q^{-\tau_i}q^{-\tau_k}}} \\
            &\leq \frac{1}{2q^{\tau_{j+d}}} + Dc'\sum_{i=1}^{d}{q^{-\tau_i}} + \frac{C(c')^2}{2}\sum_{i=1}^{d}{\sum_{k=1}^{d}{q^{-1-\tau_i-\tau_k}}} \\
            &\leq \frac{1}{2q^{\tau_{j+d}}} + Dc'dq^{-\tau_d} + \frac{C(c')^2d^2}{2}q^{-1-2\tau_d}.
\end{align*}
By choosing $c'$ small and letting $q$ be sufficiently large we have
\[Dc'd<\frac{1}{4}, \quad \text{and} \quad \frac{C(c')^2d^2}{2}q^{-1-\tau_d}<\frac{1}{4}.\]
Thus it follows that
\[\frac{1}{2q^{\tau_{j+d}}} + Dc'dq^{-\tau_d} + \frac{C(c')^2d^2}{2}q^{-1-2\tau_d}<\frac{1}{2q^{\tau_{j+d}}}+\frac{1}{4q^{\tau_d}}+\frac{1}{4q^{\tau_d}}\le \frac{1}{q^{\tau_{j+d}}}\]
where the last inequality follows from the assumption that $\tau_{d} \geq \max_{1 \leq j \leq m}{\tau_{j+d}}$.

\subsection{Completing the proof via an application of the mass transference principle for rectangles} \label{proof completion section}

Recall the Monge parameterisation of the manifold (\ref{4}). Since $\cU$ is an open subset of $\mathbb{R}^n$  which can be expressed as a countable union of closed cubes, by the countable stability of Hausdorff dimension we are free to ask $\cU$ to be of the form $\cU=\prod_{i=1}^n \cU_i$ with each $\cU_i$ being a finite closed interval in $\mathbb{R}$. Then we are in a position to complete the proof of Theorem \ref{Conjecture W} by applying the mass transference principle for rectangles
 (Theorem \ref{rectangles mtp}) to obtain a lower bound for the Hausdorff dimension of the set $W$ defined in \eqref{2}. To this end, for some suitably chosen $\ba = (a_1, a_2, \dots, a_d) \in \R_{>0}^{d}$, we will consider the set
\[W_{\ba}=\left\{x \in \cU: \left|x_i-\frac{p_i}{q}\right|<{\frac{4^{m/d}}{q^{a_i}}} \quad \text{for } 1\leq i \leq d \quad \text{for i.m. } (p_1,\dots,p_d,q) \in \cN(f,\btau)\right\}.\]
Provided that $\min_{1 \leq i \leq d}{a_i} > 1$ and \eqref{a condition} is satisfied, it follows from the preceding arguments in Section \ref{preparations section} that $W_{\ba}$ is of full Lebesgue measure in $\cU$. Thus, given $\bt = (t_1,\dots,t_d) \in \R_{\geq 0}^{d}$
we may apply Theorem \ref{rectangles mtp} to $W_{\ba}$ to obtain a lower bound for the Hausdorff dimension of the set
\begin{align*}
W_{\ba}(\bt) &= \left\{x \in \cU: \left|x_i - \frac{p_i}{q}\right|<\frac{c'}{q^{{a_i+t_i}}} \quad \text{for } 1\leq i \leq d \quad \text{for i.m. } (p_1,\dots,p_d,q) \in \cN(f,\btau)\right\},
\end{align*}
where $c'>0$ is the suitably small constant we chose earlier in Section \ref{preparations section}. In what follows, we will always choose $a_i$ and $t_i$ such that $a_i+t_i=1+\tau_i$ and so
\begin{align*}
W_{\ba}(\bt) &= \left\{x \in \cU: \left|x_i - \frac{p_i}{q}\right|<\frac{c'}{q^{1+\tau_i}} \quad \text{for } 1\leq i \leq d \quad \text{for i.m. } (p_1,\dots,p_d,q) \in \cN(f,\btau)\right\},
\end{align*}
which coincides precisely with the set $W$ from (\ref{2}) which we are interested in. Thus, any lower bound for $\dimh(W_{\ba}(\bt))$ is automatically a lower bound for $\dimh{W}$ and hence also for $\dimh{(W_{n}(\btau) \cap \cM)}$.

We are now ready to apply Theorem \ref{rectangles mtp}. We split the remainder of the proof into two cases:

\noindent \underline{{\bf Case 1:}} $\displaystyle{\tau_d \geq \frac{1-\sum_{j=1}^{m}{\tau_{j+d}}}{d}}$.

This case has already been addressed in \cite{BLW}. However, for completeness, we briefly sketch the argument here.
In this case, we let $$
a_i=1+\frac{1-\sum_{j=1}^{m}{\tau_{j+d}}}{d}, \quad \text{and} \quad t_i=(1+\tau_i)-a_i,\quad {\text{for all}} \ 1\le i\le d.
$$
Note that we have $a_i > 1$ and $a_i+t_i=1+\tau_i$ for each $1 \leq i \leq d$. Also note that
\[(a_1-1)+\dots+(a_d-1)+\sum_{j=1}^{m}{\tau_{j+d}} = d \times \left(\frac{1-\sum_{j=1}^{m}{\tau_{j+d}}}{d}\right)+\sum_{j=1}^{m}{\tau_{j+d}} = 1.\]
Hence \eqref{a condition} is satisfied.

Next, following the notation of Theorem \ref{rectangles mtp}, we consider the set
$$
\mathcal{A}=\{a_i, a_i+t_i: 1\le i\le d\}.
$$
Arranging the elements of $\cA$ in descending order, since $\tau_1 \geq \dots \geq \tau_d$, we have
$$
a_1+t_1\ge a_2+t_2\ge \cdots \ge a_d+t_d\ge a_1=\cdots=a_d.
$$

\noindent \underline{Subcase (i):} Suppose $A=a_i$ for some $1\le i\le d$. In this case, the sets $\cK_1, \cK_2, \cK_3$ appearing in Theorem \ref{rectangles mtp} corresponding to this $A$ are:
\begin{align*}
\cK_1&=\{k: a_k\ge A\}=\{1, 2,\cdots, d\},\\
\cK_2&= \{k:a_k+t_k\leq A\} \setminus \cK_1 = \emptyset, \\
\cK_3&=\{1,\dots,d\} \setminus (\cK_1 \cup \cK_2) = \emptyset.
\end{align*}
So the corresponding ``dimension number'' obtained via Theorem \ref{rectangles mtp} in this case is
\[\sum_{k \in \cK_1 \cup \cK_2}{1}+\frac{\sum_{k \in \cK_3}{a_k} - \sum_{k \in \cK_2}{t_k}}{A}=d.\]

\noindent \underline{Subcase (ii):} Suppose $A=a_i+t_i$ for some $1\le i\le d$. We may suppose that $A > a_1$, otherwise the conclusion of Subcase (i) holds. Let $1 \leq i' \leq i$ be the least index such that $a_{i'}+t_{i'} = a_i+t_i$. Then the sets $\cK_1, \cK_2, \cK_3$ appearing in Theorem~\ref{rectangles mtp} corresponding to this $A$ are:
\begin{align*}
\cK_1 &= \{k: a_k \geq A\} = \emptyset, \\[1ex]
\cK_2 &= \{k: a_k+t_k \leq A\} \setminus \cK_1 = \{i',\dots,d\}, \\[1ex]
\cK_3 &= \{1,\dots,d\}\setminus (\cK_1 \cup \cK_2) = \{1,\dots,i'-1\}.
\end{align*}

The calculation which gives the corresponding ``dimension number'' obtained via Theorem~\ref{rectangles mtp} in this case is identical to the calculation in Subcase (i) of Case 2, as given on pages \pageref{subcase 1}--15, since these cases yield the same sets $\cK_1, \cK_2, \cK_3$.

\noindent \underline{{\bf Case 2:}} $\displaystyle{\tau_d < \frac{1-\sum_{j=1}^{m}{\tau_{j+d}}}{d}}$.

In this case, let $1\le K\le d$ be the largest integer such that
\begin{align} \label{K definition}
\tau_K > \frac{1 - \sum_{j=1}^{m}{\tau_{j+d}}-(\tau_{K+1}+\dots+\tau_{d})}{K}.
\end{align}
We choose
\[a_i = \tau_i+1 \quad \text{for } K+1 \leq i \leq d,\]
and
\[a_i = \frac{1 - \sum_{j=1}^{m}{\tau_{j+d}}-(\tau_{K+1}+\dots+\tau_{d})}{K}+1 \quad \text{for } 1 \leq i \leq K.\]
We set
\[t_i = 1+ \tau_i - a_i \quad \text{for } 1 \leq i \leq d.\]
Note that we have $a_i > 1$ and $a_i+t_i=1+\tau_i$ for all $1 \leq i \leq d$ in this case. We also note that
\begin{align*}
(a_1-1)+\dots+(a_d-1)+\sum_{j=1}^{m}{\tau_{j+d}} &=  K \left(\frac{1 - \sum_{j=1}^{m}{\tau_{j+d}}-(\tau_{K+1}+\dots+\tau_{d})}{K}\right) + \sum_{k=K+1}^{d}{\tau_k} + \sum_{j=1}^{m}{\tau_{j+d}} \\
                                                 &= 1 - \sum_{j=1}^{m}{\tau_{j+d}}-(\tau_{K+1}+\dots+\tau_{d}) + \sum_{k=K+1}^{d}{\tau_k} + \sum_{j=1}^{m}{\tau_{j+d}} \\*
                                                 &=1,
\end{align*}
so \eqref{a condition} holds.

Following again the notation of Theorem \ref{rectangles mtp}, we consider the set
\[\cA = \{a_i, a_i+t_i: 1 \leq i \leq d\}.\]
Arranging the elements of $\cA$ in descending order we see that
\begin{align*}
a_1+t_1 (=\tau_1+1) \geq \dots &\geq a_K+t_K (=\tau_K+1) \\
                             &\stackrel{\eqref{K definition}}> a_1 = \dots = a_K \\
                             &> a_{K+1} = a_{K+1}+t_{K+1} \\
                             &\geq a_{K+2} = a_{K+2}+t_{K+2} \\
                             &\phantom{=}\vdots \\
                             &\geq a_d=a_d+t_d.
\end{align*}

To establish the ``dimension number'' obtained via Theorem \ref{rectangles mtp}, we will consider three possible subcases. Throughout, we will frequently use the very useful observation that
\begin{align}\label{5}t_i = 1+\tau_i-a_i = 1+\tau_i-(1+\tau_i) = 0 \quad \text{when } K+1 \leq i \leq d.\end{align}

\noindent \underline{Subcase (i):} \label{subcase 1} $A=a_i+t_i=1+\tau_i$ for some $1 \leq i \leq K$.

Suppose $A \in \cA$ is such that $A=a_i+t_i=1+\tau_i$ for some $1 \leq i \leq K$ and let $1 \leq i' \leq i$ be the least index such that $a_{i'}+t_{i'}=A$. Then the sets $\cK_1, \cK_2, \cK_3$ appearing in Theorem \ref{rectangles mtp} corresponding to this $A$ are:
\begin{align*}
\cK_1 &= \{k: a_k \geq A\} = \emptyset, \\[1ex]
\cK_2 &= \{k: a_k+t_k \leq A\} \setminus \cK_1 = \{i',\dots,d\}, \\[1ex]
\cK_3 &= \{1,\dots,d\}\setminus (\cK_1 \cup \cK_2) = \{1,\dots,i'-1\}.
\end{align*}

Thus, when $A \in \cA$ is such that $A=a_i+t_i=1+\tau_i$ for some $1 \leq i \leq K$, the corresponding ``dimension number'' obtained by using Theorem \ref{rectangles mtp} is 
\begin{align*}
\sum_{k \in \cK_1 \cup \cK_2}{1}+\frac{\sum_{k \in \cK_3}{a_k} - \sum_{k \in \cK_2}{t_k}}{A} &= (d-i'+1) + \frac{\sum_{k=1}^{i'-1}{a_k} - \sum_{k=i'}^{d}{t_k}}{a_{i}+t_{i}} \\
&=(d-i'+1) + \frac{\sum_{k=1}^{i-1}{a_k} - \sum_{k=i}^{d}{t_k}-\sum_{k=i'}^{i-1}a_k-\sum_{k=i'}^{i-1}t_k}{a_{i}+t_{i}}.
\end{align*}
By the definition of $i'$, it follows that $a_k+t_k=a_i+t_i$ for all $i'\le k<i$. Thus  
\begin{align*}
\sum_{k \in \cK_1 \cup \cK_2}{1}+\frac{\sum_{k \in \cK_3}{a_k} - \sum_{k \in \cK_2}{t_k}}{A} 
&=(d-i'+1) + \frac{\sum_{k=1}^{i-1}{a_k} - \sum_{k=i}^{d}{t_k}}{a_{i}+t_{i}}-(i-i')\\
&=(d-i+1) + \frac{\sum_{k=1}^{i-1}{a_k} - \sum_{k=i}^{d}{t_k}}{a_{i}+t_{i}} \\
&= (d-i+1)+\frac{\sum_{k=1}^{d}{a_k}-\sum_{k=i}^{d}{(t_k+a_k)}}{a_i+t_i} \\
          &\stackrel{\eqref{a condition}}= (d-i+1) + \frac{\left(d+1-\sum_{j=1}^{m}{\tau_{j+d}}\right) - \sum_{k=i}^{d}{(1+\tau_k)}}{1+\tau_i} \\
          &= (d-i+1) + \frac{i-\sum_{k=i}^{n}{\tau_k}}{1+\tau_i} \\
          & = \frac{(d-i+1)(1+\tau_i)+i-\sum_{k=i}^{n}{\tau_k}}{1+\tau_i} \\
          &= \frac{(d-i+1)(1+\tau_i)+m(1+\tau_i)-m(1+\tau_i)+i-\sum_{k=i}^{n}{\tau_k}}{1+\tau_i} \\
          &= \frac{(d+1+m)+(d-i+1+m)\tau_i-m(1+\tau_i)-\sum_{k=i}^{n}{\tau_k}}{1+\tau_i} \\
          &=\frac{(n+1)+(n-i+1)\tau_i-\sum_{k=i}^{n}{\tau_k}}{1+\tau_i}-m \\
          &= \frac{n+1+\sum_{k=i}^{n}{(\tau_i-\tau_k)}}{1+\tau_i}-m.
\end{align*}
In the penultimate line above we used the assumption that $n = d+m$.

\noindent \underline{Subcase (ii):} $A=a_i=a_i+t_i=1+\tau_i$ for some $K+1 \leq i \leq d$.

Suppose $A \in \cA$ is such that $A=a_i=a_i+t_i=1+\tau_i$ for some $K+1 \leq i \leq d$ and let $i \leq i' \leq d$ be the greatest index such that $a_{i'}=A$. Then the sets $\cK_1, \cK_2, \cK_3$ appearing in Theorem~\ref{rectangles mtp} corresponding to this $A$ are:
\begin{align*}
\cK_1 &= \{k: a_k \geq A\} = \{1,...,i'\}, \\[1ex]
\cK_2 &= \{k: a_k+t_k \leq A\} \setminus \cK_1 = \{i'+1,\dots,d\}, \\[1ex]
\cK_3 &= \{1,\dots,d\}\setminus (\cK_1 \cup \cK_2) = \emptyset.
\end{align*}
Recall (\ref{5}) that $t_k=0$ for all $K+1 \leq k \leq d$. The corresponding ``dimension number'' obtained in this case is
\[\sum_{k \in \cK_1 \cup \cK_2}{1}+\frac{\sum_{k \in \cK_3}{a_k} - \sum_{k \in \cK_2}{t_k}}{A} = d - \frac{\sum_{k=i'+1}^{d}{t_k}}{a_i+t_i} = d.\]

\noindent \underline{Subcase (iii):} $A=a_i$ for some $1 \leq i \leq K$.

If $A \in \cA$ is such that $A=a_i$ for some $1 \leq i \leq K$, then the sets $\cK_1, \cK_2, \cK_3$ appearing in Theorem~\ref{rectangles mtp} corresponding to this $A$ are:
\begin{align*}
\cK_1 &= \{k: a_k \geq A\} = \{1,...,K\}, \\[1ex]
\cK_2 &= \{k: a_k+t_k \leq A\} \setminus \cK_1 = \{K+1,\dots,d\}, \\[1ex]
\cK_3 &= \{1,\dots,d\}\setminus (\cK_1 \cup \cK_2) = \emptyset.
\end{align*}
The corresponding ``dimension number'' in this case is
\[\sum_{k \in \cK_1 \cup \cK_2}{1}+\frac{\sum_{k \in \cK_3}{a_k} - \sum_{k \in \cK_2}{t_k}}{A} = d - \frac{\sum_{k=K+1}^{d}{t_k}}{a_i+t_i} = d.\]

Thus, we conclude from these cases that
\[\dimh(W_n(\btau) \cap \cM) \geq \dimh{W} = \dimh{(W_{\ba}(\bt))} \geq \min_{1 \leq i \leq d}\left\{\frac{n+1+\sum_{k=i}^{n}{(\tau_i-\tau_k)}}{\tau_i+1}-m\right\},\]
and hence our proof of Theorem \ref{Conjecture W} is complete.

\Addresses

\end{document}